\documentclass[11pt]{article} 
\usepackage{amssymb, amsmath, latexsym, mathrsfs}
\newtheorem{defin}{}
\newtheorem{saetze}[defin]{}
\newtheorem{conjec}[defin]{}
\newtheorem{lemmas}[defin]{}
\newtheorem{folger}[defin]{}
\newtheorem{bemerk}[defin]{}

\newenvironment{theorem}  {\begin{saetze}\it {\bf Theorem:}}{\end{saetze}}

\newenvironment{lemma}    {\begin{lemmas}\it {\bf Lemma:}}{\end{lemmas}}

\newenvironment{remark}   {\begin{bemerk}\it {\bf Remark:}}{\end{bemerk}}
\newenvironment{proof}    {{\it Proof}:}{{\hfill \fillbox \bigskip}}

\newcommand{\fillbox}{\mbox{$\bullet$}}
\newcommand{\ra}{\rightarrow}

\newcommand{\ms}{\mapsto}

\newcommand{\N}{\mathbb N}
\newcommand{\F}{\mathbb F}

\newcommand{\A}{\mathcal A}

\renewcommand{\a}{\mathcal A}

\newcommand{\M}{\mathcal M}

\DeclareMathOperator{\Ann}{Ann}
\DeclareMathOperator{\nilrad}{nilrad}

\newenvironment{items}{\begin{list}{$\alph{item})$}
{\labelwidth18pt \leftmargin18pt \topsep3pt \itemsep1pt \parsep0pt}}
{\end{list}}

\newcommand{\bulit}{\item[$\bullet$]}

\begin{document}

\title{The isomorphism problem for graded algebras\\
       and its application to mod-$p$ cohomology rings\\ of small $p$-groups }
\author{Bettina Eick and Simon King}
\date{\today}
\maketitle

\begin{abstract}
The mod-$p$ cohomology ring of a non-trivial finite $p$-group is an infinite
dimensional, finitely presented graded unital algebra over the field with
$p$ elements, with generators in positive degrees. We describe an effective
algorithm to test if two such algebras are graded isomorphic. As application, 
we determine all graded isomorphisms between the mod-$p$ cohomology rings of 
all $p$-groups of order at most 100.
\end{abstract}

\section{Introduction}

The mod-$p$ cohomology ring $H^*(G, \F)$ of a non-trivial finite $p$-group 
$G$ and the field $\F$ with $p$ elements is an infinite dimensional graded
$\F$-algebra. It is an interesting and wide open question how good this
algebra is as an isomorphism invariant for the underlying group $G$. More
precisely: given two non-isomorphic $p$-groups $G$ and $H$, under which
circumstances are $H^*(G, \F)$ and $H^*(H, \F)$ isomorphic as graded
$\F$-algebras?

Our aims in this paper are two-fold. First, we consider finitely presented 
graded unital $\F$-algebras with generators in positive degrees over a finite
field $\F$; we call such algebras `finitary'. We describe an effective 
algorithm to test if two finitary algebras are graded isomorphic. We also
consider the special case of graded commutative finitary algebras and 
describe an improved algorithm for this case. 

Secondly, we apply our algorithm to the mod-$p$ cohomology rings of 
the $p$-groups with order at most $100$. The $p$-groups of order at most
$100$ are well-known, see \cite{BEO02} for a history on their classification. 
Finite presentations for their mod-$p$ cohomology rings are also available,
see \cite{Car03} and also \cite{GKi11} for a recent account.  The following
theorem exhibits a brief summary of our results. A complete list of the groups
with graded isomorphic cohomology ring is included in Section \ref{results}
below.

\begin{theorem}
The following tables list numbers of isomorphism types of groups of
order $p^n$ and numbers of graded isomorphism types of the associated 
mod-$p$ cohomology rings.
\begin{center}
\begin{tabular}{|c||c|c|c|c|c|c||c|c|c|c|}
\hline
order & 
   $2^1$ & $2^2$ & $2^3$ & $2^4$ & $2^5$ & $2^6$ & 
   $3^1$ & $3^2$ & $3^3$ & $3^4$ \\
\hline
\# groups & 1 & 2 & 5 & 14 & 51 & 267 & 1 & 2 & 5 & 15 \\
\# rings  & 1 & 2 & 5 & 14 & 48 & 239 & 1 & 2 & 5 & 15 \\
\hline
\end{tabular}
\end{center}
There are significantly more graded isomorphisms between groups of different
orders than between groups of a fixed order. In the following table we list 
numbers of isomorphism types of groups of order {\em dividing} $p^n$ and 
numbers of graded isomorphism types of the associated mod-$p$ cohomology 
rings.
\begin{center}
\begin{tabular}{|c||c|c|c|c|c|c||c|c|c|c|}
\hline
order & 
   $2^1$ & $2^2$ & $2^3$ & $2^4$ & $2^5$ & $2^6$ & 
   $3^1$ & $3^2$ & $3^3$ & $3^4$ \\
\hline
\# groups & 1 & 3 & 8 & 22 & 73 & 340 & 1 & 3 & 8 & 23 \\
\# rings  & 1 & 3 & 7 & 18 & 55 & 260 & 1 & 2 & 5 & 14 \\
\hline
\end{tabular}
\end{center}
\end{theorem}

The phenomena that there are several graded isomorphisms between mod-$p$
cohomology rings for groups of different orders is known in various examples
in the literature. A well-known example is given by the infinite families of
cyclic groups with graded isomorphic cohomology rings. Further, there are 
infinite families of metacyclic groups with graded isomorphic mod-$p$ 
cohomology rings, see \cite{Hue89}. Moreover, the result in \cite{Car05} 
implies that there are infinite families of $2$-groups of fixed coclass 
with graded isomorphic mod-$2$ cohomology rings.

\section{Preliminaries}
\label{sec:prelim}

In this preliminary section we recall some basic facts from the theory
of graded algebras and we establish our notation. Let $\N = \{0, 1, 2, 
\ldots \}$ and let $\F$ be a field. First, recall that an $\F$-algebra 
$A$ is {\em graded} if it can be written as a direct sum of $\F$-vectorspaces
\[A = \bigoplus_{n \in \N} A_n \] 
and $A_i A_j \subseteq A_{i+j}$ holds for each $i,j \in \N$. The
vectorspaces $A_0, A_1, \ldots$ are the \emph{graded components} of $A$. 
An element $a \in A$ is \emph{homogeneous}, if it is contained in a graded
component $A_n$ for some $n \in \N$; in this case we denote its degree by 
$|a|=n$.

Let $F$ be a free graded unital $\F$-algebra, let $\varphi:F\to A$ be a
surjective morphism of graded $\F$-algebras, let $\a$ be a free generating set
of $F$ and let $\mathcal R$ be a generating set of $\ker(\varphi)$. Then
$\langle \a\mid\mathcal R\rangle$ is called a graded \emph{presentation} of
$A$. The presentation is finite if both $\a$ and $\mathcal R$ are finite, and
$A$ is called \emph{finitely presented}, if a finite presentation is
given. Note that by slight abuse of notation we identify $\a$ with
$\{\varphi(x)\mid x\in \a\}$ and say that $\langle \a\mid \mathcal R\rangle$
is a presentation on the generating set $\a$ of $A$.
\medskip

{\bf Definition.} 
Let $\F$ be a finite field and let $A$ be a finitely presented graded 
unital $\F$-algebra with generators in positive degrees. Then $A$ is 
called a {\em finitary} $\F$-algebra. 

\begin{lemma}
Let $A$ be a finitary $\F$-algebra with graded components $A_0, A_1 \ldots$.
\begin{items}
\item[\rm (1)]
$A$ has a finite generating set consisting of homogeneous elements
$\A = (a_1, \ldots, a_m)$.
\item[\rm (2)]
$A$ has a finite presentation on the homogeneous generating set $\A$.
\item[\rm (3)]
Let $n \in \N$. The set $\M_n(\A) = \{a_{i_1} \cdots a_{i_j} \mid 
|a_{i_1}| + \ldots + |a_{i_j}| = n\}$ generates $A_n$ as vectorspace
and thus $A_n$ is finite dimensional.
\end{items}
\end{lemma}

\begin{proof}
(1) Each element of $A$ can be written as a finite sum of homogenous
elements. Thus each arbitrary finite generating set of $A$ gives rise
to a finite homogeneous generating set by decomposing each generator 
into homogeneous summands. \\
(2) Let $A = \langle b_1, \ldots, b_k \mid R_1, \ldots, R_l \rangle$ 
be an arbitrary finite presentation for $A$ and let $a_1, \ldots, a_m$
be an arbitrary finite generating set for $A$. Then each $b_i$ can be
written as a word in $a_1, \ldots, a_m$, say $b_i = w_i(a_1, \ldots, 
a_m)$. Similarly, each $a_j$ can be written as word in $b_1, \ldots,
b_k$, say $a_j = v_j(b_1, \ldots, b_k)$. It now follows that 
$A \cong \langle a_1, \ldots, a_m \mid R_i(w_1, \ldots, w_k) \mbox{ for }
1 \leq i \leq l \mbox{ and } a_j = v_j(w_1, \ldots, w_k) \mbox{ for }
1 \leq j \leq m \rangle$. \\
(3) Is elementary. 
\end{proof}

Let $A$ be a graded $\F$-algebra with graded components $A_0, A_1, \ldots$.
We define
\[ I(A) = \bigoplus_{n \geq 1} A_n \;\; \mbox{ and } \;\;
   I_j(A) = \bigoplus_{n \geq j} A_n \mbox{ for } j \geq 1. \]
Then $I(A)$ is called the {\em augmentation ideal} of $A$; it is a 
non-unital $\F$-algebra having the series of ideals $I(A) = I_1(A) 
\geq I_2(A) \geq \ldots$. Note that
\[ A = I(A) \rtimes A_0.\]
Let $I(A) \geq I(A)^2 \geq \ldots$ denote the series of power ideals in
$I(A)$. Then $I(A)/I(A)^c$ is a nilpotent $\F$-algebra of class $c-1$ 
for each $c \geq 1$ by construction.

\begin{lemma} \label{resnil}
Let $A$ be a finitary $\F$-algebra.
\begin{items}
\item[\rm (1)]
Let $c \in \N$. Then $I(A)^c \leq I_c(A)$.
\item[\rm (2)]
$I(A)$ is finitely generated and residually nilpotent.
\item[\rm (3)]
Let $c \in \N$. Then $I(A)/I(A)^c$ is finite dimensional.
\item[\rm (4)]
There exists $d \in \N$ with $I_{d+1}(A) \leq I(A)^2$.
\end{items}
\end{lemma}
  
\begin{proof}
(1) We use induction on $c$. For $c=1$ we note that $I(A)^1 = I(A) \leq 
I(A) = I_1(A)$. If $I(A)^c \leq I_c(A)$, then $I(A)^{c+1} = I(A) I(A)^c 
\leq I(A) I_c(A) \leq I_{c+1}(A)$. \\
(2) Let $a_1, \ldots, a_m$ be a set of homogeneous generators in positive
degrees for the unital algebra $A$. Then $I(A)$ is generated by $a_1, \ldots,
a_m$ as non-unital algebra. Thus $I(A)$ is finitely generated. Further,
$\cap_{c \geq 1} I(A)^c = \{0\}$ by (1) and thus $I(A)$ is residually 
nilpotent. \\
(3) A nilpotent quotient of a finitely generated algebra is finite
dimensional. \\
(4) By (2) the algebra $I(A)$ is finitely generated and thus it has a 
finite generating set $\A$ of homogeneous elements. Let $d$ be the maximal 
degree of a generator and let $l > d$. Then each monomial in $\M_l(\A)$ 
is a product of at least 2 elements by the definition of $d$. Hence 
$\M_l(\A) \subseteq I(A)^2$. Thus $I_{d+1}(A) = \langle \M_l(\A) \mid
l > d \rangle \subseteq I(A)^2$. 
\end{proof}

\section{Computation with finitary algebras}
\label{nilquot}

In this section we describe some elementary algorithms for finitary 
$\F$-algebras. We assume that a finitary $\F$-algebra $A$ is given
by a finite presentation $\langle a_1, \ldots, a_m \mid R_1, \ldots, 
R_l \rangle$ on homogeneous generators $\A = (a_1, \ldots, a_m)$ 
with positive degrees. We
denote the graded components of $A$ by $A_0, A_1, \ldots$ and we
assume that the Hilbert--Poincar\'e series $P_A(t) = \sum_{n \in \N} 
\dim(A_n) t^n$ is given as rational function.

It is well-known that computations with finitely presented algebraic
objects is difficult in general. For example, in the case of finitely
presented groups it is in general not algorithmically possible to 
solve the word problem (let alone the isomorphism problem). In this
section we show how this and related problems can be solved in our
considered case. For $n \geq 1$ let
\[ \epsilon_n : I(A) \ra I(A)/I(A)^n : a \ms a + I(A)^n\]
the natural epimorphism on the class-$n-1$ nilpotent quotient of $I(A)$. 
Then the image of $\epsilon_n$ is finite dimensional by Lemma \ref{resnil}.

\begin{remark}
Let $n \in \N$. Then a basis and its structure constants table for 
$I(A)/I(A)^n = Im(\epsilon_n)$ can be computed with the methods of 
\cite{Eic12} together with the images of $a_1, \ldots, a_m$ in the
finite-dimensional image. This computation requires an arbitrary
finite presentation for $I(A)$. Note that the given presentation 
$\langle a_1, \ldots, a_m \mid R_1, \ldots, R_l \rangle$ for $A$ 
defines $I(A)$ as non-unital algebra and hence a finite presentation 
for $I(A)$ is given by our setup.
\end{remark}

\subsection{The word problem}
\label{word}

Suppose that a word $w$ in the generators $\A$ is given; that is, $w = 
c + \sum_{i=1}^n c_i l_i$ with $c_i \in \F$ and $l_i \in \M_i(\A)$. Our 
aim is to decide if $w = 0$ in $A$. The following lemma translates this
to an easy calculation in the finite dimensional quotient $I(A)/I(A)^{n+1}$.

\begin{lemma}
$w = 0$ in $A$ if and only if $c_0 = 0$ and $\epsilon_{n+1}(\sum_{i=1}^n
c_i l_i) = 0$.
\end{lemma}

\begin{proof}
This follows from Lemma \ref{resnil} (1).
\end{proof}

\subsection{Bases for the graded components}
\label{homcomp}

Let $n \geq 1$ and recall that $\M_n(\A)$ generates the graded component
$A_n$ of $A$. The following lemma shows how to reduce this generating
set to a basis via a computation in the finite dimensional quotient 
$I(A)/I(A)^{n+1}$.

\begin{lemma}
$B_n$ is a basis for $A_n$ if and only if $\epsilon_{n+1}(B_n)$ is a 
basis for $\langle \epsilon_{n+1} (\M_n(\A)) \rangle$.
\end{lemma}

\begin{proof}
This follows from Lemma \ref{resnil} (1).
\end{proof}

We note that the dimensions of the graded components can be read off readily
from the Hilbert--Poincar\'e series $P_A(t)$. Define $P_A^{(0)}(t) := P_A(t)$
and $P_A^{(n)}(t) := (P_A^{(n-1)}(t) - \dim(A_{n-1}))/t$ for $n > 0$. Then 
$\dim(A_n) = P_A^{(n)}(0)$ holds.

\subsection{Detecting generating sets}
\label{gener}

Suppose that elements $b_1, \ldots, b_k$ of $I(A)$ are given. Our aim is to 
decide if these elements generate $A$ as unital algebra. The following lemma
reduces this to an elementary computation in the finite dimensional quotient
$I(A)/I(A)^2$.

\begin{lemma}
$b_1, \ldots, b_k$ generate $A$ (as unital algebra) if and only if
$\langle \epsilon_2(b_1), \ldots, \epsilon_2(b_k) \rangle = I(A)/I(A)^2$.
\end{lemma}

\begin{proof}
This follows from Lemma \ref{resnil} (2).
\end{proof}

\section{Graded isomorphisms between finitary algebras}
\label{isomtest}

In this section we exhibit our solution to the graded isomorphism
problem for finitary algebras. Recall that
two graded $\F$-algebras $A$ and $B$ are {\em graded isomorphic} if there
exists an $\F$-algebra isomorphism $\nu : A \ra B$ that respects the grading,
that is, it satisfies $\nu(A_n) = B_n$ for each $n \in \N$. We write $A
\cong B$ if $A$ is isomorphic to $B$ as $\F$-algebra and $A \cong_g B$ if
$A$ is graded isomorphic to $B$.

For our algorithm we assume that both finitary algebras $A$ and $B$ are
given by finite presentations on homogeneous generators of positive degrees
and we assume that their Hilbert--Poincar\'e series $P_A$ and $P_B$ are
available as well. We denote the graded components of $A$ and $B$ by 
$A_n$ and $B_n$, respectively.

\begin{lemma} \label{cond}
Let $A$ and $B$ be two finitary $\F$-algebras. If there exists a graded 
isomorphism $\varphi : A \ra B$, then 
\begin{items}
\item[\rm (a)]
$P_A = P_B$, and
\item[\rm (b)]
If $\A$ is a finite homogenous generating set for $A$, then $\varphi(a) 
\in B_{|a|}$ for each $a \in \A$.
\end{items}
\end{lemma}

\begin{proof}
(a) and (b) both follow from the fact that $\varphi(A_n) = B_n$ for each
$n \in \N$, where $A_n$ and $B_n$ denotes the vectorspaces in the gradings
of $A$ and $B$, respectively.
\end{proof}

Lemma \ref{cond} (b) shows that there are only finitely many possible 
options for graded isomorphisms $A \ra B$, since a finite homogenous
generating set for $A$ is given and $B_n$ is finite for each $n \in \N$.

\begin{theorem} \label{check}
Let $A$ and $B$ be two finitary $\F$-algebras and suppose that $P_A = P_B$.
Let $A = \langle a_1, \ldots, a_m \mid R_1, \ldots, R_l \rangle$ a finite
homogenous presentation for $A$ on generators of positive degree and let 
$b_1, \ldots, b_m \in B$ with $b_i \in B_{|a_i|}$ for $1 \leq i \leq m$. 
The map $\varphi : A \ra B : a_i \ms b_i$ extends to a graded isomorphism 
if and only if
\begin{items}
\item[\rm (a)]
$R_j(b_1, \ldots, b_m) = 0$ for $1 \leq j \leq l$, and
\item[\rm (b)]
$b_1, \ldots, b_m$ generate $B$.
\end{items}
\end{theorem}

\begin{proof}
First suppose that $\varphi$ extends to a graded isomorphism. Then
$0 = \varphi(0) = \varphi( R_j(a_1, \ldots, a_m)) = R_j(\varphi(a_1),
\ldots, \varphi(a_n)) = R_j(b_1, \ldots, b_n)$ and thus (a) holds. 
(b) is obvious. 

Now suppose that (a) and (b) hold. Then (a) yields that $\varphi$ is
an algebra homomorphism. As $b_i \in B_{|a_i|}$ for $1 \leq i \leq m$,
it follows that $\varphi$ respects the grading and $\varphi(A_n) \subseteq
B_n$ for $n \in \N$. (b) asserts that $\varphi$ is surjective. Hence
$\varphi(A_n) = B_n$ for each $n \in \N$. Finally, as $P_A = P_B$, we obtain 
that $\varphi$ is also injective and hence a graded isomorphism.
\end{proof} 

Lemma \ref{cond} and Theorem \ref{check} induce the following method to 
determine a graded isomorphism $A \ra B$ if it exists. Let $A = \langle
a_1, \ldots, a_m \mid R_1, \ldots, R_l \rangle$ be a finite homogenous
presentation on generators of positive degrees and let $d_i = |a_i|$ for 
$1 \leq i \leq m$. 
\bigskip

{\bf GradedIsomorphism}($A$, $B$) 
\begin{items}
\item[(1)] 
Test if $P_A = P_B$; if not, then return false.
\item[(2)]
Determine bases for $B_{d_1}, \ldots, B_{d_m}$.
\item[(3)]
For each $(b_1, \ldots, b_m) \in B_{d_1} \times \ldots \times B_{d_m}$ do
\begin{items}
\item[(a)]
Check that $R_j(b_1, \ldots, b_m) = 0$ for $1 \leq j \leq l$.
\item[(b)]
Check that $b_1, \ldots, b_m$ generate $B$.
\item[(c)]
If (a) and (b) are satisfied, then return $(b_1, \ldots, b_m)$.
\end{items}
\item[(4)]
Return false;
\end{items}
\bigskip

Note that bases for $B_{d_1}, \ldots, B_{d_m}$ can be determined as in
Section~\ref{homcomp}. Each of these spaces is finite and thus the 
for-loop in Step (3) is a finite loop.  Step (3a) can be implemented by 
the method in Section~\ref{word}.  Step (3b) can be performed as in
Section~\ref{gener}.

\begin{remark} \label{limit}
Let $w_1 = \max \{ d_i \mid 1 \leq i \leq m\}$ and let $w_2$ denote
the maximal degree of a monomial in $R_1, \ldots, R_l$. Further, let
$w = \max \{ 1, w_1, w_2 \}$. Then the algorithm {\bf GradedIsomorphism}
requires the computation of $\epsilon_{w+1}$.
\end{remark}

If $I(A)/I(A)^{w+1}$ and $I(B)/I(B)^{w+1}$ are both available, then this
allows further reductions in the algorithm {\bf GradedIsomorphism}. For
example, if $A$ and $B$ are graded isomorphic, then $dim(I(A)/I(A)^c) 
= dim(I(B)/I(B)^c)$ for each $c \geq 1$ and this induces an additional
condition that may be checked in Step (1) of the algorithm for all
available nilpotent quotients. Further, if $a_i \in I(A)^{c_i}$ for some
$c_i \in \N$, then $\varphi(b_i) \in I(B)^{c_i} \cap B_{d_i}$. This
can be used to obtain a reduction in Step (3) of the algorithm. 

\section{The graded commutative case}
\label{grcomm}

A graded $\F$-algebra $A$ is called \emph{graded commutative}, if for all
homogeneous elements $x,y\in A$ the equation $x\cdot y = (-1)^{|x|\cdot
|y|}y\cdot x$ holds. If $char(\F) = 2$, then a graded commutative algebra
is commutative and the free graded commutative $\F$-algebra on $m$ generators
is isomorphic to the polynomial ring on $m$ generators. If $char(\F) > 2$,
then a free graded commutative $\F$-algebra is isomorphic to a tensor 
product of a polynomial ring and an exterior algebra.

We present graded commutative $\F$-algebras not as quotients of free graded
unital $\F$-algebras (as in Section~\ref{sec:prelim}), but as quotients of
free graded commutative $F$-algebras. Hence, if $F$ is a free graded
commutative $\F$-algebra, and $\varphi:F\to A$ is a surjective morphism of
graded $\F$-algebras, and $\a$ is a free generating set of $F$ and $\mathcal
R$ is a generating set of $K = \ker(\varphi)$, then we call $\langle
\a\mid\mathcal R\rangle$ a \emph{graded commutative presentation} of
$A$. Note that one can choose $\mathcal R$ so that its elements are
homogeneous.

Finitely presented graded commutative algebras are noetherian and are either
commutative or are non-commutative $G$-algebras, for which a Gr\"obner basis
theory is available much similar to the commutative case \cite[Chapter
1.9]{GP08}. With Gr\"obner bases, one has an alternative way to solve
computational problems than by using nilpotent quotients as in Sections
\ref{nilquot} and \ref{isomtest}. That approach can be more effective; in
particular, this is the case if the parameter $w$ as determined in Remark
\ref{limit} is large.

Let $\langle a_1, \ldots, a_m \mid R_1, \ldots, R_l \rangle$ be a finite
graded commutative presentation of $A$ corresponding to $\varphi : F \ra A$
with $K = \ker(\varphi)$, as above. We  consider a \emph{Gr\"obner basis}
$\mathcal B = (B_1, \ldots, B_k)$ for $K$.
\begin{items}
\bulit
The word problem in $A$ can be solved by polynomial reduction with 
respect to $\mathcal B$, and a basis of $A_n$ is given by those elements of $\M_n(A)$ that are not 
divisible by any of the leading monomials of $B_1,...,B_k$.
\bulit
By~\cite[Proposition 3.6.6 d)]{KRo05}, the computation of Gr\"obner bases 
also allows for an effective test whether a subset of $A$ forms a generating 
set of $A$.
\bulit
The Hilbert--Poincar\'e series $P_A(t)$ can be computed as in~\cite[Chapter
5.2]{GP08} or~\cite[Chapter 5]{KRo05},  and is a rational
function.
More generally, if $I\leq A$ is an ideal generated by homogeneous 
elements, then the quotient ring $A/I$ is finitary graded commutative, and its
Hilbert--Poincar\'e series $P_{A/I}(t)$ (to which we also refer to as the
``Hilbert--Poincar\'e series of $I$'') can be computed, too.
\bulit
The nilradical of $A$ can in principle be computed as in~\cite[Chapter
4.5]{KRo05}. There is a more efficient alternative approach is available
for cohomology rings.
If $G$ is a finite group, then the nilradical of $H^*(G, \mathbb F_p)$ is
formed by the elements that have nilpotent restriction to all the maximal 
$p$-elementary abelian subgroups of $G$, by a result of Quillen (see also 
\cite[Theorem 8.4.3]{Car05}). Based on this, the nilradicals of modular 
cohomology rings of finite groups can be
computed by intersecting the preimages of certain explicitly given ideals
under morphisms (namely restrictions) of finitely presented graded commutative
$\F$-algebras. The preimages can be computed as in \cite[Section 1.8.10,
Remark 1.8.17]{GP08}, and their intersection as in~\cite[Section 1.7.7]{GP08}.
\bulit
If $I\leq A$ is an ideal generated by homogeneous elements, then its 
annihilator $\Ann(I)=\{x\in A \mid \forall y\in I: y\cdot x=0\}\leq A$ 
can be computed~\cite[Section 2.8.4]{GP08}.
\end{items}

When we test in Step~(3)(b) whether elements $b_1,\ldots, b_n\in B$ generate
$B$ according to \cite[Proposition 3.6.6 d)]{KRo05}, then the computation of a
Gr\"obner basis \emph{in elimination order} for an ideal defined in terms of
$(b_1,\ldots,b_n)$ is needed. This is potentially a very expensive
operation. It is thus crucial to reduce the possible choices of
$(b_1,\ldots,b_n)$ in Step~(3) by other methods, as described in the rest of
this section. Here, elimination is used as well, but it turns out that this is
feasible and reduces the computation time drastically.

\subsection{Early detection of non-isomorphic algebras}
\label{sec:earlyreductions}

Comparing $P_{A}(t)$ and $P_{B}(t)$ as in Step~(1) of Algorithm
\textbf{GradedIsomorphism} allows to disprove the existence of a graded
isomorphism between $A$ and $B$ in many cases. In addition to that, we compute
the nilradicals $\nilrad(A)$ and $\nilrad(B)$ of $A$ and $B$, and test if
$P_{A/\nilrad(A)}(t)=P_{B/\nilrad(B)}(t)$. This may detect that $A\not\cong_g
B$ even in cases where $P_{A}(t)=P_{B}(t)$.

\subsection{Reducing the list of potential generator images}
\label{sec:cutbranches}

We now focus on possible reductions of the images $(b_1,\ldots,b_n)$ of
$(a_1,\ldots, a_n)$ to be considered in Step~(3) of Algorithm
\textbf{GradedIsomorphism}.

Let $(\a_A, {\mathcal R}_A)$ and $(\a_B, {\mathcal R}_B)$ be finite graded
commutative presentations of $A$ and $B$.  Let $\hat\a_A = (a_{i_1},\ldots,
a_{i_k})$ be a subset of $\a_A$, and let $b_{i_1},...,b_{i_k}\in B$. Let
$I=\langle \hat\a_A\rangle\leq A$ be the ideal generated by $\hat\a_A$, and
$J=\langle b_{i_1},...,b_{i_k}\rangle\leq B$. We discuss here three tests that
often allow to conclude that there is no graded homomorphism mapping $a_{i_j}$
to $b_{i_j}$ for $j=1,...,k$.

Firstly, if $\varphi: A\to B$ is a graded isomorphism, then
$P_{A/I}(t)=P_{B/\varphi(I)}(t)$. Hence, if the ideals $I\leq A$ and $J\leq
B$ have different Hilbert--Poincar\'e series, then the map $a_{i_j}\mapsto
b_{i_j}$ can not be extended to a graded isomorphism.

Secondly, by elimination of the variables $\a_A\setminus \hat\a_A$ from the
relation ideal $\mathcal R_A$ as in \cite[Section 1.8.2]{GP08}, one obtains
relations $\hat R_{A,1},...,\hat R_{A,l}$ that only involve elements of
$\hat\a_A$. If $\varphi: A\to B$ is a graded homomorphism, then $\hat
R_{A,c}\left(\varphi(a_{i_1}),...,\varphi(a_{i_k})\right)=0$ for all
$c=1,...,l$, which can be effectively tested using a Gr\"obner basis of
$\langle {\mathcal R}_B\rangle$. Hence, if $\hat R_{A,c}(b_{i_1},...,b_{i_k})\not=0$ for
some $c=1,...,l$, then the map $a_{i_j}\mapsto b_{i_j}$ can not be extended to
a graded homomorphism.

Thirdly, one can compute the \emph{annihilators} $\Ann\left(I\right) \leq A$
and $\Ann\left(J\right)\leq B$. If they have different Hilbert--Poincar\'e
series, then the map $a_{i_j}\mapsto b_{i_j}$ can not be extended to a graded
isomorphism. In principle, it would also be possible to compare the radicals
of the two ideals, but we found that this does not help to improve efficiency.

How do the above tests to help simplify in Step~(3) of Algorithm
\textbf{GradedIsomorphism}? It suffices to restrict Step~(3) to
those tuples $(b_1,\ldots,b_n)\in B_{d_1}\times...\times B_{d_n}$ that pass
the above three tests for all subsets of $(a_1,\ldots, a_n)$. In a practical
implementation, one would start by using the three tests on one-element
subsets, \emph{i.e.}, one would compute all possible elements of $B_{d_i}$
that may occur as images of $a_i$ under any graded algebra isomorphism, for
$i=1,...,n$. This will normally leave very few possibilities, say, $\hat
B_{d_i}\subseteq B_{d_i}$. Next, one would consider all possible pairs
$(b_i,b_j)\in \hat B_{d_i}\times \hat B_{d_j}$, and use the three tests to
determine all possible images of $(a_i,a_j)$ under any graded isomorphism, for
$i,j=1,...,n$. And so on, with larger subsets of $\a_A$.

If $A\not\cong_g B$, the three tests will often leave no or only very few
possible choices for $(b_1,...,b_n)$ in Step~(3) of Algorithm
\textbf{GradedIsomorphism} that need to be tested in Step~(3)(b). And if
$A\cong_g B$, then often the first possible choice of $(b_1,...,b_n)$ will
turn out to yield a graded isomorphism by the final test in Step~(3)(b).

\begin{remark}
  One should be aware that computing the nilradicals of $A$ and $B$, a graded
  commutative presentation of $\langle\langle\hat A\rangle\rangle$, or
  annihilators, can generally be computationally very expensive. However, in
  all the examples that we considered, the gain of using the additional
  reductions in Algorithm \textbf{GradedIsomorphism} outweighs these
  additional costs by far.
\end{remark}

\section{Examples}
\label{sec:examples}

We summarise here some examples of cohomology rings of finite groups. A
minimal graded commutative presentation for each ring has been computed with
the optional \texttt{pGroupCohomology} package~\cite{GKi13} for
Sage~\cite{Sage}. The package uses Singular~\cite{Singular} for the
computation of Gr\"obner bases, annihilators and elimination in graded
commutative rings. Nilradicals are computed as described in
Section~\ref{results}.

\subsection{Early detection of non-isomorphy}

Let $A$ be the mod-$3$ cohomology ring of the extraspecial $3$-group of order
$27$ and exponent 3, which is group number $3$ of order $27$ in the small
groups library~\cite{SmallGroups}. Let $B$ be the mod-$3$ cohomology ring of
the Sylow $3$-subgroup of $U_3(8)$, which is group number $9$ of order $81$ in
the small groups library. Each of these algebras has a minimal graded
commutative presentations with generators in degrees $1$, $1$, $2$, $2$, $2$,
$2$, $3$, $3$, and $6$.

The Hilbert-Poincar\'e series of the two algebras, respectively of their
nilradicals, coincide. The power series expansion of the Hilbert-Poincar\'e
series is \[ P_{A}(t)=P_{B}(t) = 1 + 2t + 4t^{2} + 6t^{3} + 7t^{4} + 8t^{5} +
9t^{6} + 10t^{7} + 12t^{8} + \cdots\] Thus, $B_1$, $B_2$, $B_3$ and $B_6$
contain $3^2-1=8$, $3^4-1=80$, $3^6-1=728$ and $19682$ non-zero elements,
respectively. Hence, without the reductions from
Section~\ref{sec:cutbranches}, one would need to consider $8^2\cdot 80^4\cdot
728^2\cdot 19682>10^{19}$ possible images for the generators of $A$.

However, it turns out that there is one degree-$2$ generator $a\in A$ so that
$P_{A/\langle a\rangle}(t)$ is different from $P_{B/\langle b\rangle}(t)$, for
each of the $80$ non-zero elements of $B_2$. Hence we can readily detect that
$A$ and $B$ are not graded isomorphic.

\subsection{A more difficult to detect pair of non-isomorphic algebras}

Let $A$ be the mod-$2$ cohomology ring of group number $27$ of order $32$ in
the small groups library, and let $B$ be the mod-$2$ cohomology ring of group
number $128$ of order $64$. They both have minimal graded commutative
presentations formed by three generators in degree $1$ and three generators in
degree $2$, and four relations.

The Hilbert-Poincar\'e series of the algebras, respectively of their
nilradicals, coincide. The power series expansion of the Hilbert-Poincar\'e
series is
\[ P_A(t) = 1 + 3t + 7t^{2} + 13t^{3} + 22t^{4} + 34t^{5} + 50t^{6} + \cdots\]
Thus, $B_1$ contains $2^3-1=7$ and $B_2$ contains $2^7-1=127$ non-zero
elements. Hence, without the reductions from Section~\ref{sec:cutbranches},
one would need to consider $7^3\cdot 127^3>7\cdot 10^8$ possible images for
the generators of $A$.

In contrast to the previous example, the methods from
Section~\ref{sec:cutbranches} applied to one-element subsets of the generating
set of $A$ are not strong enough to prove $A\not\cong_g B$. However, when
applied to the triple of degree-$1$ generators, the tests only leave $6$
candidates for the images of the triple under isomorphism. Applied to the
three degree-$1$ and two of the degree-$2$ generators, still as many as $4608$
different isomorphic images seem possible. And thus one needs to combine each
of them with the $111$ potential isomorphic images of the remaining degree-$2$
generator. In all but $176$ cases, the mapping of generators does not extend
to a homomorphism, and in the remaining $176$ cases the homomorphism is not
surjective. Hence, the two algebras are not graded isomorphic.

\section{Application to cohomology rings}
\label{results}

Let $p$ be a prime, let $G$ be a finite $p$-group and let $\F$ be the
field with $p$ elements. Then the mod-$p$ cohomology ring $H^*(G, \F)$
is a graded $\F$-algebra defined by
\[ H^*(G, \F) = \bigoplus_{n \in \N} H^n(G, \F). \]

By the theorem of Evens--Venkov (see also \cite[Theorem 6.5.1]{Car05}),
modular cohomology rings of finite groups are finitely presentable
graded-commutative algebras. Each graded component $H^n(G, \F)$ is a finite
dimensional vectorspace over $\F$ and $H^0(G, \F) \cong \F$. Hence 
$H^*(G, \F)$ is a finitary $\F$-algebra. The methods of \cite{GKi11} 
determine a \emph{minimal} presentation of $H^*(G, \F)$ and these allow
to apply the methods described in the first part of this paper.


In the following sections we exhibit the graded isomorphisms among the 
mod-$p$ cohomology rings of the $p$-groups of order at most $100$. As
a preliminary step we observe that the rank of the underlying $p$-group
is an isomorphism invariant for the cohomology ring. Recall that the
rank of a finite $p$-group $G$ is the rank of the finite elementary 
abelian quotient $G/[G,G]G^p = G/\phi(G)$ of $G$ or, equivalently, the 
minimal generator number of $G$.

In the following we consider the $p$-groups of order at most $100$ by
their generator number.  The groups with $1$ generator are the cyclic 
groups; it is well-known that the cyclic of order $p^n$ have graded
isomorphic mod-$p$-cohomology rings (with the exception of the cyclic 
group of order $2$). We thus omit this case in our list below. It then
remains to consider the groups of order dividing $2^6$, $3^4$, $5^2$ 
and $7^2$. The cases $5^2$ and $7^2$ are again well-understood and hence 
we focus on $2^6$ and $3^4$ in the following exposition.

\subsection{$2$-groups}

We give here a complete and irredundant list of all graded isomorphic mod-$2$
cohomology rings $H^*(G, \F)$ for the groups $G$ of order dividing $64$. We
identify a group $G$ by its id $[ order, number]$ in the SmallGroups library,
see \cite{SmallGroups}.

Each of the following lists of groups statisfies that the mod-$2$ cohomology
rings of the considered groups are pairwise graded isomorphic, and mod-$2$
cohomology rings of groups from different lists are not graded isomorphic. If
a group of order dividing $64$ does not appear in any of the lists, then
the graded isomorphism type of its mod-$2$ cohomology ring is unique among all
groups of order dividing $64$.

Additionally to the ids of the groups in the list, we include the rank of
the groups and in many cases also a structure description. For the latter,
we denote with $C_k,D_k, Q_k, SD_k$ the cyclic, dihedral, 
quaternion and semidihedral groups of order $k$, respectively. The symbols 
$\times$, $:$ and $.$ describe a direct product, a split extension and an
arbitrary extension, respectively.

If one of the groups in one of the following lists is metacyclic, then all
groups are metacyclic and we include this information as well. We note that 
our result differ in one case from the theoretical description of the mod-$p$
cohomology rings of metacyclic groups in \cite{Hue89}: our results imply
that the mod-$2$ cohomology rings of the metacyclic groups $[32, 15]$ and 
$[64,49]$ are graded isomorphic to each other, but they are not graded 
isomorphic to $[64, 45]$ as Theorem E(2) of \cite{Hue89} suggests. This
is based on the fact that the presentations of the mod-$p$ cohomology rings
of $[32, 15]$ and $[64, 49]$ as given in \cite{GKi11} and also in 
\cite[Appendix]{Car03} are not compatible with that in \cite{Hue89}; for
example, the presentations obtained in \cite{GKi11} and in 
\cite[Appendix]{Car03} imply that the underlying cohomology rings have 
a non-nilpotent element in degree $3$ in contradiction to Theorem E(2) of
\cite{Hue89}.

\begin{items}
\bulit groups $[ 4, 1 ], [ 8, 1 ], [ 16, 1 ], [ 32, 1 ], [ 64, 1 ]$ \\ 
rank 1 and cyclic

\bulit groups $[ 8, 2 ], [ 16, 5 ], [ 32, 16 ], [ 64, 50 ]$ \\ 
rank 2, metacyclic, and structure $C_{2^n} \times C_2$ ($n>1$)

\bulit groups $[ 8, 3 ], [ 16, 7 ], [ 32, 18 ], [ 64, 52 ]$ \\ 
rank 2, metacyclic, and structure $D_{2^n}$ ($n>2$)

\bulit groups $[ 16, 2 ], [ 32, 3 ], [ 32, 4 ], [ 64, 2 ], [ 64, 3 ], 
  [ 64, 26 ], [ 64, 27 ]$ \\ 
rank 2, metacyclic, and structure $C_{2^n} : C_{2^m}$ ($n,m>1$)

\bulit groups $[ 16, 3 ], [ 32, 9 ], [ 64, 38 ]$ \\ 
rank 2, metacyclic, and structure $(C_{2^n} \times C_2) : C_2$ ($n>1$)

\bulit groups $[ 16, 4 ], [ 32, 12 ], [ 32, 13 ], [ 32, 14 ], 
[ 64, 15 ], [ 64, 16 ], [ 64, 44 ], [ 64, 47 ], [ 64, 48 ]$ \\ 
rank 2, metacyclic, and structure $C_{2^m} : C_{2^n}$ ($m,n>1$)

\bulit groups $[ 16, 6 ], [ 32, 17 ], [ 64, 51 ]$ \\ 
rank 2, metacyclic, and structure $C_{2^n} : C_2$ ($n>2$)

\bulit groups $[ 16, 8 ], [ 32, 19 ], [ 64, 53 ]$ \\ 
rank 2, metacyclic, and structure $SD_{2^n}$ ($n>3$)

\bulit groups $[ 16, 9 ], [ 32, 20 ], [ 64, 54 ]$ \\ 
rank 2, metacyclic, and structure $Q_{2^n}$ ($n>3$)

\bulit groups $[ 16, 10 ], [ 32, 36 ], [ 64, 183 ]$ \\ 
rank 3, and structure $C_{2^n} \times C_2 \times C_2$ ($n>1$)

\bulit groups $[ 16, 11 ], [ 32, 39 ], [ 64, 186 ]$ \\ 
rank 3, and structure $C_2 \times D_{2^n}$ ($n>2$)

\bulit groups $[ 32, 2 ], [ 64, 17 ], [ 64, 21 ]$ \\ 
rank 2, and structure $(C_{2^n} \times C_2) : C_4$ ($n>1$)

\bulit groups $[ 32, 5 ], [ 64, 6 ], [ 64, 29 ]$ \\ 
rank 2, and structure $(C_{2^n} \times C_{2^m}) : C_2$ ($n>m$)

\bulit groups $[ 32, 10 ], [ 64, 39 ]$ \\ 
rank 2, and structure $Q_{2^n} : C_4$ ($n>2$)

\bulit groups $[ 32, 15 ], [ 64, 49 ]$  \\ 
rank 2, metacyclic, and structure $C_4 . D_{2^n}$ ($n>2$)

\bulit groups $[ 32, 21 ], [ 64, 83 ], [ 64, 84 ]$ \\ 
rank 3, and structure $(C_{2^n} : C_4) \times C_2$ ($n>1$)

\bulit groups $[ 32, 22 ], [ 64, 95 ]$ \\ 
rank 3, and structure $((C_{2^n} : C_2) : C_2) \times C_2$ ($n>1$)

\bulit groups $[ 32, 23 ], [ 64, 103 ], [ 64, 106 ], [ 64, 107 ]$ \\ 
rank 3, and structure $(C_{2^m} : C_{2^n}) \times C_2$ ($m,n>1$)

\bulit groups $[ 32, 25 ], [ 64, 115 ], [ 64, 118 ], [ 64, 123 ]$ \\ 
rank 3

\bulit groups $[ 32, 26 ], [ 64, 126 ]$ \\ 
rank 3, and structure $C_{2^n} : Q_8$ ($n>1$)

\bulit groups $[ 32, 28 ], [ 64, 140 ], [ 64, 147 ]$ \\ 
rank 3

\bulit groups $[ 32, 29 ], [ 64, 155 ], [ 64, 157 ]$ \\ 
rank 3

\bulit groups $[ 32, 31 ], [ 64, 167 ]$ \\ 
rank 3, and structure $(C_{2^n} \times C_4) : C_2$ ($n>1$)

\bulit groups $[ 32, 34 ], [ 64, 174 ]$ \\ 
rank 3, and structure $(C_{2^n} \times C_4) : C_2$ ($n>1$)

\bulit groups $[ 32, 35 ], [ 64, 181 ]$ \\ 
rank 3, and structure $C_{2^n} : Q_8$ ($n>1$)

\bulit groups $[ 32, 37 ], [ 64, 184 ]$ \\ 
rank 3, and structure $(C_{2^n} : C_2) \times C_2$ ($n>2$)

\bulit groups $[ 32, 38 ], [ 64, 185 ]$ \\ 
rank 3, and structure $(C_{2^n} \times C_2) : C_2$ ($n>2$)

\bulit groups $[ 32, 40 ], [ 64, 187 ]$ \\ 
rank 3, and structure $C_2 \times SD_{2^n}$ ($n>3$)

\bulit groups $[ 32, 41 ], [ 64, 188 ]$ \\ 
rank 3, and structure $C_2 \times Q_{2^n}$ ($n>3$)

\bulit groups $[ 32, 42 ], [ 64, 189 ]$ \\ 
rank 3, and structure $(C_{2^n} : C_2) \times C_2$ ($n>2$)

\bulit groups $[ 32, 43 ], [ 64, 190 ]$ \\ 
rank 3, and structure $(C_2 \times D_{2^n}) : C_2$ ($n>2$)

\bulit groups $[ 32, 44 ], [ 64, 191 ]$ \\ 
rank 3, and structure $(C_2 \times Q_{2^n}) : C_2$ ($n>2$)

\bulit groups $[ 32, 45 ], [ 64, 246 ]$ \\ 
rank 4, and structure $C_{2^n} \times C_2 \times C_2 \times C_2$ ($n>1$)

\bulit groups $[ 32, 46 ], [ 64, 250 ]$ \\ 
rank 4, and structure $C_2 \times C_2 \times D_{2^n}$ ($n>2$)

\bulit groups $[ 64, 13 ], [ 64, 14 ]$ \\ 
rank 2

\bulit groups $[ 64, 31 ], [ 64, 40 ]$ \\ 
rank 2, and structure $(C_{16} \times C_{2}) : C_{2}$

\bulit groups $[ 64, 112 ], [ 64, 113 ]$ \\ 
rank 3, and structure $(C_{4} : C_{8}) : C_{2}$ 

\bulit groups $[ 64, 119 ], [ 64, 121 ]$ \\ 
rank 3

\bulit groups $[ 64, 120 ], [ 64, 122 ]$ \\ 
rank 3, and structure $Q_{16} : C_4$

\bulit groups $[ 64, 124 ], [ 64, 125 ]$ \\ 
rank 3

\bulit groups $[ 64, 142 ], [ 64, 148 ]$ \\ 
rank 3

\bulit groups $[ 64, 144 ], [ 64, 146 ]$ \\ 
rank 3

\bulit groups $[ 64, 156 ], [ 64, 158 ]$ \\ 
rank 3, and structure $Q_8 : Q_8$ 

\bulit groups $[ 64, 161 ], [ 64, 162 ]$ \\ 
rank 3, and structure $(C_2 \times (C_4 : C_4)) : C_2$

\bulit groups $[ 64, 164 ], [ 64, 165 ]$ \\ 
rank 3, and structure $(Q_8 : C_4) : C_2$

\bulit groups $[ 64, 173 ], [ 64, 176 ]$ \\ 
rank 3, and structure $(C_8 \times C_4) : C_2$
\end{items} 

\subsection{$3$-groups}

\begin{items}
\bulit
groups $[3,1], [ 9, 1 ], [ 27, 1 ], [ 81, 1 ]$ \\
rank 1 and cyclic
\bulit
groups $[ 9, 2 ], [ 27, 2 ], [ 81, 2 ], [ 81, 4 ], [ 81, 5 ]$ \\
rank 2, metacyclic and structure $C_{3^n} : C_{3^m}$
\bulit
groups $[ 27, 4 ], [ 81, 6 ]$ \\
rank 2, and structure $C_{3^n} : C_3$
\bulit
groups $[ 27, 5 ], [ 81, 11 ]$\\
rank 3, and structure $C_{3^n} : C_3^2$
\end{items}


\def\cprime{$'$}

\end{document}